\documentclass[12pt]{article}
\usepackage[margin=1in]{geometry}
\usepackage[usenames]{xcolor}
\usepackage{amssymb}
\usepackage{relsize}
\usepackage{amsmath}
\usepackage{amsthm}
\usepackage{amsfonts}
\usepackage{amscd}
\usepackage{graphicx}
\usepackage{hyperref}
\usepackage{tikz-cd}
\usepackage{tikz}

\author{Damanvir Singh Binner\thanks{Indian Institute of Science Education and
	Research (IISER), Mohali, India.  damanvirbinnar@iisermohali.ac.in } 
 \and Neha Gupta\thanks{Shiv Nadar University, Noida, India. neha.gupta@snu.edu.in}  
\and Manoj Upreti\thanks{Shiv Nadar University, Noida, India. mu506@snu.edu.in } }

\begin{document}

\theoremstyle{plain}
\newtheorem{theorem}{Theorem}
\newtheorem{corollary}[theorem]{Corollary}
\newtheorem{lemma}[theorem]{Lemma}
\newtheorem{proposition}[theorem]{Proposition}
\newtheorem{question}[theorem]{Question}

\theoremstyle{definition}
\newtheorem{definition}[theorem]{Definition}
\newtheorem{example}[theorem]{Example}
\newtheorem{conjecture}[theorem]{Conjecture}

\theoremstyle{remark}
\newtheorem{remark}[theorem]{Remark}

\title{\Large\bf{Berkovich-Uncu type Partition Inequalities Concerning Impermissible Sets and Perfect Power Frequencies}}

\date{}
\maketitle

\begin{abstract}
Recently, Rattan and the first author (Ann. Comb. $25$ ($2021$) $697$--$728$) proved a conjectured inequality of Berkovich and Uncu (Ann. Comb. $23$ ($2019$) $263$--$284$) concerning partitions with an impermissible part. In this article, we generalize this inequality upon considering $t$ impermissible parts. We compare these with partitions whose certain parts appear with a frequency which is a perfect $t^{th}$ power. Our inequalities hold after a certain bound, which for given $t$ is a polynomial in $s$, a major improvement over the previously known bound in the case $t=1$. To prove these inequalities, our methods involve constructing injective maps between the relevant sets of partitions. The construction of these maps crucially involves concepts from analysis and calculus, such as explicit maps used to prove countability of $\mathbb{N}^t$, and Jensen's inequality for convex functions, and then merge them with techniques from number theory such as Frobenius numbers, congruence classes, binary numbers and quadratic residues.  We also show a connection of our results to colored partitions. Finally, we pose an open problem which seems to be related to power residues and the almost universality of diagonal ternary quadratic forms. 
\end{abstract}

\section{Introduction}
\label{S1}

Though inequalities between certain classes of integer partitions have been studied for a long time, they have recently received special attention (\cite{Grizzell, Lovejoy, James, Chern, Saikia}). Moreover, several recent studies have focussed on sets of partitions whose parts come from some interval \cite{abr, berkbook, chapman}. Working in both of these directions, Berkovich and Uncu conjectured some intriguing inequalities \cite[Conjecture 3.2, Conjecture 3.3, Conjecture 7.1]{BU19} regarding the relative sizes of two closely related sets consisting of integer partitions. These conjectures were proven independently by Zang and Zeng \cite{zang}, and by Rattan and the first author \cite{BR20}. While the former researchers approached these conjectures using partly analytic and partly combinatorial methods, the latter used entirely combinatorial methods. A detailed comparison between the two approaches can be found in \cite[Section 1.1]{BR20}. We describe the main result of Binner and Rattan which appears in \cite[Theorem 3]{BR20}. For positive integers $L \geq 3$, $s$ and $k$, with $s+1 \leq k \leq L+s$,
\begin{itemize}
\item $I_{L,s,k}$ is the set of
partitions where the smallest part is $s$, all parts are $\leq L+s$, and $k$ does
not appear as a part.
\item $D_{L,s}$  denotes the set of nonempty partitions
		with parts in the set $\{s+1,  \ldots,
		L+s\}$. 
\end{itemize}

\begin{theorem}[Binner and Rattan (2021)]
\label{MostGen}
For positive integers $L$, $s$ and $k$, with $L \geq 3$ and $s+1 \leq k \leq
L+s$, we have    
\begin{equation*}
        |\{\pi \in I_{L,s,k} : |\pi| = N \}| > |\{\pi \in D_{L,s} : |\pi| =
        N\}|,
    \end{equation*}
for all $N \geq \Gamma(s)$, where $\Gamma(s)$ is defined in \cite[(15)]{BR20}.
    \end{theorem}
    
 At this point, the precise value of $\Gamma(s)$ is not important. However, Theorem \ref{MostGen} has been stated with the constant $\Gamma(s)$ inserted to emphasize that it is explicitly known and only depends on $s$. We also mention that the bound $\Gamma(s)$ in Theorem \ref{MostGen} is huge, in fact of the order $O((6s)^{(6s)^{18s}})$.

Whenever a part cannot occur from a range of allowable parts, as with $k$ in the
definition of $I_{L,s,k}$, we refer to that as an \emph{impermissible part}. In the present article, we generalize Theorem \ref{MostGen} by looking at the effect of considering an impermissible set $V \subset \{s+1, \ldots, L+s\}$ of elements, instead of an impermissible part $k$ in the definition of the set $I_{L,s,k}$. 

\begin{itemize}
\item $I_{L,s,V}$ is the set of
partitions where the smallest part is $s$, all parts are $\leq L+s$, and the elements of $V$ do
not appear as a part.
    \end{itemize}
    
    For $V = \{k_1, k_2, \ldots, k_t\}$, it is clear that $I_{L,s,V} \subset I_{L,s,k_1}$, and thus in view of Theorem \ref{MostGen}, it is natural to ask the following question. 

\begin{question}
\label{Q2}
For any $V \subset \{s+1, \ldots, L+s\}$, does there exist a bound $M$, which only depends on $s$ and $|V|$ such that for $N \geq M$,
 \begin{equation*}
        |\{\pi \in I_{L,s,V} : |\pi| = N \}| > |\{\pi \in D_{L,s} :
        |\pi| = N\}|.
    \end{equation*}
\end{question}

In the proof of Theorem \ref{MostGen}, the chief strategy was to remove all parts of $k$ and compensate by adding appropriate parts of $s$ and some other elements. To ensure injectivity of the map, the frequency of $s$ in the image was chosen in such a manner that one could recover the frequency of $k$ in the original partition. A natural approach is to try to generalize the proof of Theorem \ref{MostGen} to answer the above question in the affirmative. 

We would then need to construct an injective map in which we remove all members of $V$ (along with their frequencies), and compensate by adding appropriate parts of $s$. This map should be such that one can recover the frequencies of all members of $V$ from just the frequency of $s$ in the image. The existence of such an injective map seems too ambitious. The following theorem answers Question \ref{Q2} in the negative. 

\begin{theorem}
\label{T1}
For any $V \subset \{s+1, \ldots, L+s\}$, which is not a singleton set, and for all $N \geq 2(L+s)^7 + (L+s)^5$, 
 \begin{equation*}
        |\{\pi \in I_{L,s,V} : |\pi| = N \}| \leq |\{\pi \in D_{L,s} :
        |\pi| = N\}|.
    \end{equation*}
\end{theorem}
  
   We prove Theorem \ref{T1} in Section \ref{S2}. It is natural to wonder whether for Theorem \ref{T1}, one could find a bound which only depends on $s$. However, this is not possible, because if there exists a bound $M$ depending only on $s$, we can choose $L$ to be much larger than the bound $M$. If we choose a set $V = \{k_1, k_2, \ldots, k_t\}$ with $t > 1$, such that $k_1, k_2, \cdots k_{t-1}$ are all larger than the bound $M$, then we have  $$|\{\pi \in I_{L,s,V} : |\pi| = M \}| =  |\{\pi \in I_{L,s,k_t} : |\pi| = M \}| > |\{\pi \in D_{L,s} : |\pi| = M\}|,$$ giving the required contradiction.
   
   The negative answer to Question \ref{Q2} suggests that we might need to consider an appropriate subset of $D_{L,s}$ to generalize Theorem \ref{MostGen} for any impermissible set $V$. Basically, we need to remove all $f_1$ parts of $k_1$, $f_2$ parts of $k_2$, \ldots, $f_t$ parts of $k_t$, and compensate by adding $f$ parts of $s$ and some other parts. In particular, we need the following inequality to hold 
   \begin{equation}
   \label{Inequality}
    k_1f_1 + k_2f_2 + \cdots + k_tf_t \geq sf.
    \end{equation}
     Moreover, we should be able to recover the values of $f_1$, $f_2$, \ldots $f_t$ from the value of $f$ alone. That is, $f \in \mathbb{N}$ should be an expression in $f_1 \in \mathbb{N}$, $f_2 \in \mathbb{N}$, \ldots $f_t \in \mathbb{N}$, such that the value of $f$ uniquely determines the value of $(f_1, f_2, \ldots, f_t)$. This suggests that we need an injective map $\mathbb{N}^t \rightarrow \mathbb{N}$ such that $(f_1,f_2, \ldots, f_t) \mapsto f$. Since $ \mathbb{N}^t$ is countable, there exist such injective maps. We need their explicit description. The motivation comes from the case $t=2$. Recall the famous Cantor's injective map $ \mathbb{N} \times \mathbb{N} \rightarrow \mathbb{N}$, described in the following diagram.
   
   
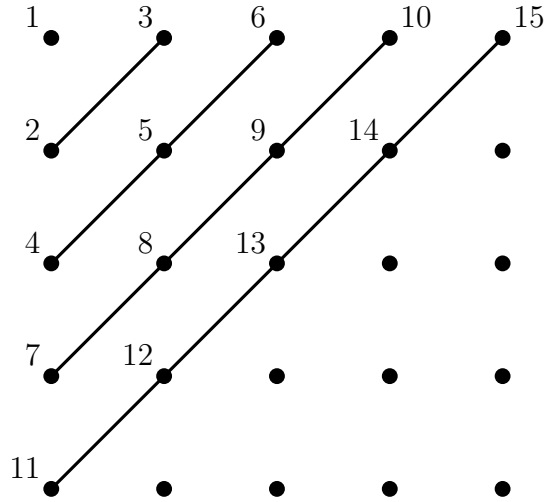
\vspace{.5cm}
   
     \begin{center}
   \begin{tikzpicture}

\fill[black!80!black] (.75,2.25) circle (3pt);
\fill[black!80!black] (2.25,2.25) circle (3pt);
\fill[black!80!black] (3.75,2.25) circle (3pt);
\fill[black!80!black] (5.25,2.25) circle (3pt);
\fill[black!80!black] (6.75,2.25) circle (3pt);

\fill[black!80!black] (.75,3.75) circle (3pt);
\fill[black!80!black] (2.25,3.75) circle (3pt);
\fill[black!80!black] (3.75,3.75) circle (3pt);
\fill[black!80!black] (5.25,3.75) circle (3pt);
\fill[black!80!black] (6.75,3.75) circle (3pt);

\fill[black!80!black] (.75,5.25) circle (3pt);
\fill[black!80!black] (2.25,5.25) circle (3pt);
\fill[black!80!black] (3.75,5.25) circle (3pt);
\fill[black!80!black] (5.25,5.25) circle (3pt);
\fill[black!80!black] (6.75,5.25) circle (3pt);

\fill[black!80!black] (.75,6.75) circle (3pt);
\fill[black!80!black] (2.25,6.75) circle (3pt);
\fill[black!80!black] (3.75,6.75) circle (3pt);
\fill[black!80!black] (5.25,6.75) circle (3pt);
\fill[black!80!black] (6.75,6.75) circle (3pt);

\fill[black!80!black] (.75,8.25) circle (3pt);
\fill[black!80!black] (2.25,8.25) circle (3pt);
\fill[black!80!black] (3.75,8.25) circle (3pt);
\fill[black!80!black] (5.25,8.25) circle (3pt);
\fill[black!80!black] (6.75,8.25) circle (3pt);

\draw[solid,black!50!black,very thick] (.75,6.75)-- (2.25,8.25);
\draw[solid,black!50!black,very thick] (.75,5.25)-- (2.25,6.75)--(3.75,8.25);
\draw[solid,black!50!black,very thick] (.75,3.75)-- (2.25,5.25)--(3.75,6.75)--(5.25,8.25);
\draw[solid,black!50!black,very thick] (.75,2.25)-- (2.25,3.75)--(3.75,5.25)--(5.25,6.75)--(6.75,8.25);

\node at (.75,8.25) [above left] {$1$};

\node at (.75,6.75) [above left] {$2$};
\node at (2.25,8.25) [above left] {$3$};

\node at (.75,5.25) [above left] {$4$};
\node at (2.25,6.75) [above left] {$5$};
\node at (3.75,8.25) [above left] {$6$};

\node at (.75,3.75) [above left] {$7$};
\node at (2.25,5.25) [above left] {$8$};
\node at (3.75,6.75) [above left] {$9$};
\node at (5.25,8.25) [above right] {$10$};

\node at (.75,2.25) [above left] {$11$};
\node at (2.25,3.75) [above left] {$12$};
\node at (3.75,5.25) [above left] {$13$};
\node at (5.25,6.75) [above left] {$14$};
\node at (6.75,8.25) [above right] {$15$};

\end{tikzpicture}
\end{center}

 \vspace{.75cm}

   Algebraically, the above bijection $\mathbb{N}^2 \rightarrow \mathbb{N}$ can be expressed as $$ (m,n) \mapsto {m+n-1 \choose 2} + m. $$ To be able to easily generalize this map to $t$ dimensions, we make a slight modification, and consider the following injective map $\mathbb{N}^2 \rightarrow \mathbb{N}$ given by $$ (m,n) \mapsto {m+n \choose 2} + m. $$ If we use the latter map, \eqref{Inequality} suggests that we need the following inequality to hold 
   \begin{equation}
   \label{Inequality2}
    k_1 f_1 + k_2 f_2 \geq s \left({f_1+f_2 \choose 2} + f_1\right). 
    \end{equation}
     Since the right hand side of \eqref{Inequality2} is a quadratic polynomial in $f_1$ and $f_2$, the inequality \eqref{Inequality2} will not hold for large values of $f_1$ and $f_2$. This suggests us to impose some conditions on the frequencies $f_1$ of $k_1$ and $f_2$ of $k_2$ in order to ensure that the inequality \eqref{Inequality2} holds. To make the left hand side a quadratic polynomial too, we suppose that $f_1$ and $f_2$ are perfect squares. That is, let $f_1 = m^2$ and $f_2 = n^2$ for some $m, n \geq 0$. In other words, we compare the set $I_{L,s,k}$ to the subset $D_{L,s,k_1,k_2} \subset D_{L,s}$, whose frequencies of $k_1$ and $k_2$ are perfect squares. Then, the following inequality holds for all $k_1, k_2 > s$ and all $m,n \in \mathbb{N}$ $$ k_1 m^2 + k_2 n^2 \geq s \left({m+n \choose 2} + m\right). $$ Thus, one could remove the $m^2$ parts of $k_1$ and $n^2$ parts of $k_2$, and compensate by adding ${m+n \choose 2} + m$ parts of $s$ and some other parts. Further, one could recover the values of $m$ and $n$ from the frequency of $s$ in the image, ensuring that the resultant map is injective. To generalize the above procedure to $t$ dimensions, it is natural to try the map $\mathbb{N}^t \rightarrow \mathbb{N}$ given by 
     \begin{equation}
     \label{Oneone}
      (m_1,m_2, \ldots, m_t) \mapsto {m_1 + m_2 + \cdots + m_t \choose t} + \cdots + {m_1 + m_2 \choose 2} + {m_1 \choose 1}. 
      \end{equation}
      The injectivity of the map in \eqref{Oneone} follows from the concept of combinatorial number system (see \cite{Becken, McC, Sid}). 
Generalizing the subset $D_{L,s,k_1,k_2}$ of $D_{L,s}$ to $t$ dimensions, we consider the following refinement of $D_{L,s}$:

 \begin{itemize}
\item $D_{L,s,V}$  denotes the set of nonempty partitions
		with parts in the set $\{s+1,  \ldots,
		L+s\}$, such that the frequencies of the members of $V$ are perfect $t^{th}$ powers $(0,1, 2^{t}, 3^{t}, \ldots)$, where $t = |V|$.
\end{itemize}

Then, it is natural to ask the following question, which if true provides an elegant generalization of Theorem \ref{MostGen}. 

\begin{question}
\label{T2}
For any $V \subset \{s+1, \ldots L+s\}$, does there exist a bound $M$, which only depends on $s$ and $|V|$ such that for $N \geq M$,
 \begin{equation*}
        |\{\pi \in I_{L,s,V} : |\pi| = N \}| \geq |\{\pi \in D_{L,s,V} : |\pi| = N\}|.
    \end{equation*}
\end{question}

Suppose $V = \{k_1, k_2, \ldots, k_t\}$. Without loss of generality, we can assume that $k_1 > k_2 > \cdots > k_t$. Let $f_i$ denote the frequency of $k_i$. Since each $f_i$ is a perfect $t^{th}$ power, we have $f_i = m_i^t$ for some $m_i \geq 0$. To generalize our method for $t=2$ described above, we would need the following inequality to hold. 
\begin{equation}
\label{Needed}
 k_1 m_1^t + k_2 m_2^t + \cdots k_t m_t^t \geq s \left( {m_1 + m_2 + \cdots + m_t \choose t} + \cdots + {m_1 + m_2 \choose 2} + {m_1 \choose 1} \right). 
 \end{equation}
 We consider the values of $k_i$ for which the inequality \eqref{Needed} holds for all values of $m_i$. The insight for such an equality comes from the following result which is proved using the theory of convex functions, especially Jensen's inequality. 
\begin{lemma}
\label{Comb}
For natural numbers $m_1, m_2, \ldots m_t$, we have $$ m_1^t + m_2^t + \cdots m_t^t \geq \frac{t!}{t^t} \left({m_1 + m_2 + \cdots + m_t \choose t} + \cdots + {m_1 + m_2 \choose 2} + {m_1 \choose 1}\right). $$
\end{lemma}

Thus, whenever $ k_t > \frac{t^t}{t!}s$, the required inequality in \eqref{Needed} holds. However, for small values of $k_t$, the inequality in \eqref{Needed} may not hold. In Section \ref{S4}, we discuss Lemma \ref{Comb} and use it to answer Question \ref{T2} in the case $k_t \geq  \frac{2^{t+4}t^t}{t!}s + s^2$. Another difficulty that we overcome in this proof is the case when some of the $m_i$'s are zero. To handle these cases, we use the concept of binary numbers and congruence classes to ensure that the map we construct is injective. 

As mentioned above, the desired inequality \eqref{Needed} fails if $k_1$, $k_2$, \ldots $k_t$ are small. We resolve this issue by working with an altogether different injective map $\mathbb{N}^t \rightarrow \mathbb{N}$. Our motivation again comes from the case $t=2$, in which another known diagram gives an elegant bijection $\phi: \mathbb{N}^2 \rightarrow \mathbb{N}$. 

 
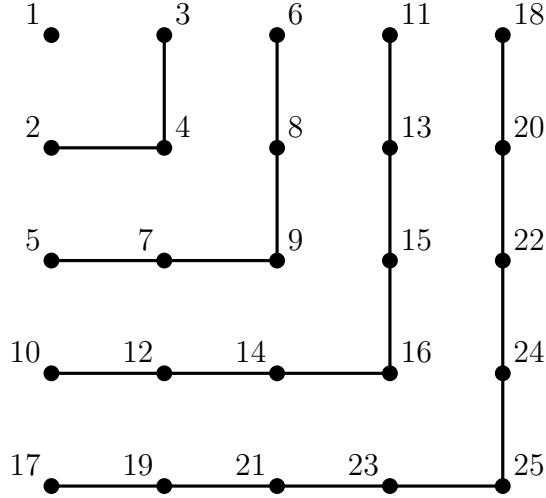
\vspace{.5cm}

\begin{center}
\begin{tikzpicture}
\fill[black!80!black] (3,3) circle (3pt);
\fill[black!80!black] (4.5,3) circle (3pt);
\fill[black!80!black] (6,3) circle (3pt);
\fill[black!80!black] (7.5,3) circle (3pt);
\fill[black!80!black] (9,3) circle (3pt);

\fill[black!80!black] (3,4.5) circle (3pt);
\fill[black!80!black] (4.5,4.5) circle (3pt);
\fill[black!80!black] (6,4.5) circle (3pt);
\fill[black!80!black] (7.5,4.5) circle (3pt);
\fill[black!80!black] (9,4.5) circle (3pt);

\fill[black!80!black] (3,6) circle (3pt);
\fill[black!80!black] (4.5,6) circle (3pt);
\fill[black!80!black] (6,6) circle (3pt);
\fill[black!80!black] (7.5,6) circle (3pt);
\fill[black!80!black] (9,6) circle (3pt);

\fill[black!80!black] (3,7.5) circle (3pt);
\fill[black!80!black] (4.5,7.5) circle (3pt);
\fill[black!80!black] (6,7.5) circle (3pt);
\fill[black!80!black] (7.5,7.5) circle (3pt);
\fill[black!80!black] (9,7.5) circle (3pt);

\fill[black!80!black] (3,9) circle (3pt);
\fill[black!80!black] (4.5,9) circle (3pt);
\fill[black!80!black] (6,9) circle (3pt);
\fill[black!80!black] (7.5,9) circle (3pt);
\fill[black!80!black] (9,9) circle (3pt);

\draw[solid,black!50!black,very thick] (3,7.5)-- (4.5,7.5)--(4.5,9);
\draw[solid,black!50!black,very thick] (3,6)-- (4.5,6)--(6,6)--(6,7.5)--(6,9);
\draw[solid,black!50!black,very thick] (3,4.5)--(4.5,4.5)--(6,4.5)--(7.5,4.5)--(7.5,6)--(7.5,7.5)--(7.5,9);
\draw[solid,black!50!black,very thick] (3,3)-- (4.5,3)--(6,3)--(7.5,3)--(9,3)--(9,4.5)--(9,6)--(9,7.5)--(9,9);

 \node at (3,9) [above left] {$1$};
\node at (3,7.5) [above left] {$2$};
\node at (4.5,9) [above right] {$3$};
\node at (4.5,7.5) [above right] {$4$};

\node at (3,6) [above left] {$5$};
\node at (6,9) [above right] {$6$};
\node at (4.5,6) [above left] {$7$};
\node at (6,7.5) [above right] {$8$};
\node at (6,6) [above right] {$9$};

\node at (3,4.5) [above left] {$10$};
\node at (7.5,9) [above right] {$11$};
\node at (4.5,4.5) [above left] {$12$};
\node at (7.5,7.5) [above right] {$13$};
\node at (6,4.5) [above left] {$14$};
\node at (7.5,6) [above right] {$15$};
\node at (7.5,4.5) [above right] {$16$};

\node at (3,3) [above left] {$17$};
\node at (9,9) [above right] {$18$};
\node at (4.5,3) [above left] {$19$};
\node at (9,7.5) [above right] {$20$};
\node at (6,3) [above left] {$21$};
\node at (9,6) [above right] {$22$};
\node at (7.5,3) [above left] {$23$};
\node at (9,4.5) [above right] {$24$};
\node at (9,3) [above right] {$25$};
\end{tikzpicture}
\end{center}

\vspace{.5cm}

To write out explicitly, the map $\phi: \mathbb{N}^2 \rightarrow \mathbb{N}$ is such that if $m \geq n$, $$ (m,n) \mapsto (m-1)^2 + 2n-1,$$  and if $n > m$, $$ (m,n) \mapsto (n-1)^2 + 2m. $$ Thus, in either case, we have $$(\max(m,n)-1)^2 < \phi(m,n) \leq (\max(m,n))^2. $$ We generalize this idea to iteratively construct a bijective map $\psi_0: \mathbb{N}^t \rightarrow \mathbb{N}$. First map the point $(1,1,\ldots,1)$ to $1$. Inductively, suppose $\psi_0$ has been defined for all the points inside the $t$-dimensional cube of side $h$ using the numbers $1, 2, \ldots, h^t$. Then, we can map all the remaining points inside the $t$-dimensional cube of side $h+1$ in any order, using the numbers $h^t+1, h^t + 2, \ldots, (h+1)^t$. The most helpful feature of this map is the property 
\begin{equation}
\label{Best1}
(\max(m_1,m_2,\ldots,m_t)-1)^t < \psi_0(m_1,m_2,\ldots,m_t) \leq (\max(m_1,m_2,\ldots,m_t))^t. 
\end{equation}

In contrast to the previous map, this one can also easily deal with some of the $m_i$'s being zero. For that we slightly modify the map $\psi_0$. Let $\mathbb{W}$ denote the set of whole numbers. We just replace $m_i$ by $m_i-1$ in the above map to get a bijection $\psi: \mathbb{W}^t \rightarrow \mathbb{N}$ with the property 
\begin{equation}
\label{Best2}
 (\max(m_1,m_2,\ldots,m_t))^t < \psi(m_1,m_2,\ldots,m_t) \leq (\max(m_1,m_2,\ldots,m_t)+1)^t. 
 \end{equation}

Using the property in \eqref{Best2}, and extending some ideas in the proof of \cite[Theorem 11]{BR20}, we answer Question \ref{T2} in the affirmative, provided the set $V$ satisfies the condition $|V| \leq \frac{L-2}{2}$. This condition is imposed to ensure that the set $\{s+1, \ldots, L+s\}$ contains two consecutive numbers that are not in the set $V$. Finding consecutive numbers that are not in $V$ ensures the availability of coprime numbers that are not in $V$, and then one could apply the concept of Frobenius number, described in Lemma \ref{Frobenius} below. 

\begin{theorem}
\label{T3}
For any $V \subset \{s+1, \ldots, L+s\}$ with $|V| \leq \frac{L-2}{2}$, and $N \geq (10s)^5 (t+1)^4 (39s^2t^3)^{5t}$ (where $t = |V|$),
 \begin{equation*}
        |\{\pi \in I_{L,s,V} : |\pi| = N \}| \geq |\{\pi \in D_{L,s,V} : |\pi| = N\}|.
    \end{equation*}
\end{theorem}

We prove Theorem \ref{T3} in Section \ref{S4}. Apart from generalizing Theorem \ref{T1} to $t$-dimensions, another important achievement of Theorem \ref{T3} is a great improvement in the bound on $N$. The bound in Theorem \ref{MostGen} was huge, in fact of the order $O((6s)^{(6s)^{18s}})$. In contrast, for a given $t$, the bound on $N$ in Theorem \ref{T3} is just a polynomial in $s$. 

In Section \ref{S3}, we give another proof of Theorem \ref{T3} for $k_t \geq \left(2^{t+4}s+s^2\right) \frac{t^t}{t!}$ using Lemma \ref{Comb}, along with the concept of binary numbers and congruence classes. The main strength of this alternate proof over the one in Section \ref{S4} is that for $k_t \geq \left(2^{t+4}s+s^2\right) \frac{t^t}{t!}$, it leads to a much smaller bound on $N$ after which the desired inequality of Theorem \ref{T3} holds. The bound is in fact less than $(15s)^5$, and is independent of $t$. In Section \ref{S5}, we describe the implications of Theorem \ref{T1} and Theorem \ref{T3} to positivity and negativity of certain $q$-series. These $q$-series results show an interesting connection to partitions whose parts in $V$ can appear in two colours, as described in Section \ref{S5}. In Section \ref{S6}, we pose two open problems which seem to be connected to power residues and the almost universality of ternary diagonal quadratic forms. 

As in the proof of Theorem \ref{MostGen}, we heavily rely on the concept of Frobenius numbers, described in the following lemma. 

\begin{lemma}[Sylvester (1882)]
\label{Frobenius}
For natural numbers $a$ and $b$ such that $\gcd(a,b)=1$, the equation
$ax+by=n$ has a solution $(x, y)$, with $x$ and $y$ nonnegative
integers, whenever $n \geq (a-1)(b-1)$. 
\end{lemma} 

Lemma \ref{Frobenius} shows that the largest number that cannot be expressed in the form $ax+by$, known as the Frobenius number of $a$ and $b$, is equal to $ab-a-b$. Sylvester \cite{Sylvester82} proved Lemma \ref{Frobenius} in $1882$. For more contemporary proofs, we refer the reader to the four proofs in \cite[Pages 31-34]{Alfonsin}. The following refinement of Lemma \ref{Frobenius} will be particularly useful for us.

\begin{corollary}
\label{Refined}
Let $a$, $b$ and $n$ be natural numbers, and $h$ be a nonnegative integer. Suppose $\gcd(a,b)$ divides $n$ and $n \geq (a-1)(b-1) + abh$. Then the equation $ax+by=n$ has a nonnegative integer solution $(x, y)$ such that $bh \leq x < b(h+1)$.
\end{corollary}

\begin{proof}
Let $g=\gcd(a,b)$ and $n' = n-abh$. Then $g$ divides $n'$, and $\gcd(\frac{a}{g},\frac{b}{g})=1$. Further, from $n' \geq (a-1)(b-1)$, it easily follows that $$ \frac{n'}{g} \geq \left(\frac{a}{g}-1\right) \left(\frac{b}{g}-1\right). $$ 
Therefore, using Lemma \ref{Frobenius}, the equation $$ \frac{n'}{g} = \left(\frac{a}{g}\right) x + \left(\frac{b}{g} \right) y $$ has a nonnegative integer solution $(x,y)$. Clearing denominators of the above equation, we get a nonnegative integer solution $(x,y)$ for the equation $n'=ax+by$. In fact, we can choose a solution $(x_0,y_0)$ with $0 \leq x_0 < b$, because if $(x,y)$ is a solution to $ax+by=n'$, then for any $l \in \mathbb{N}$, $(x-lb,y+la)$ is also a solution. Finally, since $n=n'+abh$, $(x_0 + bh,y_0)$ is a solution to $ax+by=n$ satisfying $bh \leq x_0+bh < b(h+1)$. 
\end{proof}

\section{Proof of Theorem \ref{T1}}
\label{S2}

\begin{proof}

Since $V$ is not a singleton set, it has at least $2$ elements, say $k_1$ and $k_2$, with $k_1 > k_2$. We denote $\gcd(k_1,k_2)$ by $d$. For $N \geq 2(L+s)^7 + (L+s)^5$, we construct an injective map
 \begin{equation*}
	\phi : \{\pi \in I_{L,s,V} : |\pi| = N\} \rightarrow \{\pi \in D_{L,s} : |\pi| = N \}.
\end{equation*} 

Let $ \pi = \left(s^{f_s}, \ldots, k_2^0, \ldots, k_1^0, \ldots, (L+s)^{f_{L+s}} \right) $. We denote $f_s$ by $f$. Note that $f \geq 1$. We make two cases based on whether $f \geq k_1^3$ or not. Each case will have two further subcases. To prove that the map $\phi$ is injective, we ensure that $\phi(\pi)$ has different frequencies of $k_2$ in different cases, as described in Table \ref{tab1}, and then prove that $\phi$ is injective within each case. 

\begin{table}[htpb]
    \centering
    \begin{tabular}{|c|c|}
 \hline
Case      &   Possible frequencies of $k_2$ \\
\hline
1(a)        &  $\{0,1,\ldots,k_1-1\}$ \\
\hline
1(b)     &   $\{k_1, k_1+1, \ldots, dk_1-1\}$ \\
\hline
2(a)        &  $\{dk_1,dk_1+1, \ldots, dk_1+k_1^4-1\}$   \\
\hline
2(b)        &   $\{dk_1+k_1^4, dk_1+k_1^4+1, \ldots, dk_1+2k_1^4-1\}$ \\
\hline
    \end{tabular}
    \caption{The possible frequencies of $k_2$ in $\phi(\pi)$ in different cases}
        \label{tab1}
\end{table}

Case $1$: Suppose $f > k_1^3$. Let $\alpha_f$ denote the remainder when $sf$ is divided by $d$. Note that $0 \leq \alpha_f < d$. We consider two further subcases based on whether $\alpha_f = 0$ or not.

Case $1$(a): Suppose $\alpha_f = 0$. Then $d$ divides $sf$. Since $f > k_1^3$, by Corollary \ref{Refined}, the equation 
\begin{equation}
\label{Later}
 sf = k_2x_f + k_1y_f 
 \end{equation}
  has a nonnegative integer solution $(x_f, y_f)$ satisfying $$ 0 \leq x_f < k_1.$$ For every $f$, fix such a solution $(x_f,y_f)$. Then, define the map $\phi$ as follows:  $$ \phi(\pi) = \left(s^0, (s+1)^{f_{s+1}}, \ldots, k_2^{x_f}, \ldots, k_1^{y_f}, \ldots, (L+s)^{f_{L+s}}   \right). $$ Note that from $\phi(\pi)$, we can recover the values of $x_f$ and $y_f$, and then using the defining equation \eqref{Later}, we can recover $f$, proving that the map $\phi$ is injective within Case $1$(a).
  
  Case $1$(b): Suppose $0 < \alpha_f < d$. Since $d$ divides $k_1$ and $k_2$, $d$ also divides $k_1-k_2$. In particular, $d \leq k_1-k_2$. Therefore, $$k_2 < k_2 + \alpha_f < k_1.$$ Using $f > k_1^3$, it is easy to verify that $$ sf- (k_2+\alpha_f) \geq (k_1-1)(k_2-1) + k_1k_2 \alpha_f. $$ Further, $sf-(k_2 + \alpha_f)$ is divisible by $d$. Thus, by Corollary \ref{Refined}, it follows that the equation 
\begin{equation}
\label{Later'}
 sf - (k_2 + \alpha_f) = k_2x_f + k_1y_f 
 \end{equation}
  has a nonnegative integer solution $(x_f, y_f)$ satisfying $$ \alpha_f k_1 \leq x_f < (\alpha_f+1) k_1.$$ For every $f$, fix such a solution $(x_f,y_f)$. Then, define the map $\phi$ as follows:  $$ \phi(\pi) = \left(s^0, (s+1)^{f_{s+1}}, \ldots, k_2^{x_f}, \ldots, (k_2 + \alpha_f)^{f_{k_2+\alpha_f}+1}, \ldots, k_1^{y_f}, \ldots, (L+s)^{f_{L+s}}   \right). $$ From $\phi(\pi)$, we can recover the values of $x_f$ and $y_f$. Then, we can find $\alpha_f$ using the property $$ \alpha_f = \left \lfloor \frac{x_f}{k_1} \right \rfloor. $$
  Finally, using the values of $x_f$, $y_f$ and $\alpha_f$ in \eqref{Later'}, we can get the value of $f$, proving that the map $\phi$ is injective within Case $1$(b).  

Case $2$: Suppose $1 \leq f \leq k_1^3$. Since $N \geq 2(L+s)^7 + (L+s)^5$ is large enough, there exists some $s+1 \leq i \leq L+s$ such that 
\begin{equation}
\label{1000}
f_i \geq 2k_1^5 + k_1^3. 
\end{equation}
Let $i_0$ be the least such number. Note that the above bound on $N$ is not sharp for the inequality \eqref{1000} to hold but it has been stated this way to avoid messier terms. Let $\sigma_{f,i_0}$ denote the quantity $sf + i_0 \left(2k_1^5 + k_1^3\right)$, and   $\beta_{f,i_0}$ be the remainder when $\sigma_{f,i_0}$ is divided by $d$. Note that $0 \leq \beta_{f,i_0} < d$. We consider two further subcases based on whether $\beta_{f,i_0} = 0$ or not.

Case $2$(a): Suppose $\beta_{f,i_0} = 0$. Then $d$ divides $\sigma_{f,i_0}$. By Corollary \ref{Refined}, an easy calculation shows that the equation 
\begin{equation}
\label{Later''}
 sf + i_0 \left(2k_1^5 + k_1^3\right) = k_2x_{f,i_0} + k_1y_{f,i_0} 
 \end{equation}
  has a nonnegative integer solution $(x_{f,i_0}, y_{f,i_0})$ satisfying $$ k_1(d+f-1) \leq x_{f,i_0} < k_1(d+f).$$ For every $f$ and $i_0$, fix such a solution $(x_{f,i_0}, y_{f,i_0})$. Then, define the map $\phi$ as follows:  $$ \phi(\pi) = \left(s^0, (s+1)^{f_{s+1}}, \ldots, k_2^{x_{f,i_0}}, \ldots, k_1^{y_{f,i_0}}, \ldots, i_0^{f_{i_0}-(2k_1^5 + k_1^3)}, \ldots, (L+s)^{f_{L+s}} \right),$$ where it is understood that the part $i_0$ is not precisely placed (it may, for example, be the case that $i_0 < k_2$). From $\phi(\pi)$, we can recover the values of $x_{f,i_0}$ and $y_{f,i_0}$. Then, we can find $f$ using the property $$ d+f-1 = \left \lfloor \frac{x_{f,i_0}}{k_1} \right \rfloor. $$
   Finally, using the values of $x_{f,i_0}$, $y_{f,i_0}$ and $f$ in \eqref{Later''}, we can get the value of $i_0$, proving that the map $\phi$ is injective within Case $2$(a). 
   
   Case $2$(b): Suppose $0 <  \beta_{f,i_0} < d$. Note that $$k_2 < k_2 + \beta_{f,i_0} < k_1.$$ Since $sf-(k_2 + \beta_{f,i_0})$ is divisible by $d$, by Corollary \ref{Refined}, an easy calculation shows that the equation 
\begin{equation}
\label{Later'''}
 sf + i_0 \left(2k_1^5 + k_1^3\right) - (k_2 +  \beta_{f,i_0}) = k_2x_{f,i_0} + k_1y_{f,i_0} 
 \end{equation}
  has a nonnegative integer solution $(x_{f,i_0}, y_{f,i_0})$ satisfying $$ k_1(d+k_1^3+f-1) \leq x_{f,i_0} < k_1(d+k_1^3+f).$$ For every $f$ and $i_0$, fix such a solution $(x_{f,i_0}, y_{f,i_0})$. Then, define the map $\phi$ as follows:  $$ \phi(\pi) = \left(s^0, (s+1)^{f_{s+1}}, \ldots, k_2^{x_{f,i_0}}, \ldots, k_1^{y_{f,i_0}}, \ldots, i_0^{f_{i_0}-(2k_1^5 + k_1^3)},  \ldots, (L+s)^{f_{L+s}} \right), $$ where it is understood that the part $i_0$ is not precisely placed. From $\phi(\pi)$, we can recover the values of $x_{f,i_0}$ and $y_{f,i_0}$. Then, we can find $f$ using the property $$ d+k_1^3+f-1 = \left \lfloor \frac{x_{f,i_0}}{k_1} \right \rfloor. $$
   Further, using the values of $x_{f,i_0}$, $y_{f,i_0}$ and $f$ in \eqref{Later'''}, we get the value of $i_0\left(2k_1^5 + k_1^3\right) - (k_2 +  \beta_{f,i_0})$. Since $k_2 +  \beta_{f,i_0} < k_1 < 2k_1^5 + k_1^3$, we get $$ i_0 - 1 = \left \lfloor \frac{k_2x_{f,i_0} + k_1y_{f,i_0} - sf}{2k_1^5+k_1^3} \right \rfloor, $$ giving the value of $i_0$. For given $f$ and $i_0$, we already know the value of $\beta_{f,i_0}$, proving that the map $\phi$ is injective within Case $2$(b). 
   
   For the overall injectivity, note that in different cases, the frequency of $k_2$ in $\phi(\pi)$ is different, as described in Table \ref{tab1}. 
\end{proof}

\section{Proof of Theorem \ref{T3}}
\label{S3}

We crucially use the bijection $\psi$ described in Section \ref{S1}, and its important property described in \eqref{Best2}. For brevity of notation, we often denote $\psi(m_1,m_2,\ldots,m_t)$ by $\psi$. Thus, in this notation, the value of $\psi$ determines the values of $m_1$, $m_2$, \ldots, and $m_t$. We will require the following lemma in the proof of Theorem \ref{T3}. 

\begin{lemma}
\label{Crucial1}
Suppose $h$, $s$ and $t$ be natural numbers with $h \geq 2st^2$. Then, $$ \left(1+\frac{1}{h}\right)^t \leq 1+\frac{1}{2s}. $$
\end{lemma}

\begin{proof}
Since $h \geq 2st^2$, we have $$ \left(1+\frac{1}{h}\right)^t \leq \left(1+\frac{1}{2st^2}\right)^t. $$ Let $f(t):= \left(1+\frac{1}{2st^2}\right)^t$ be viewed as a function of $t$. Then $f(1) = 1+\frac{1}{2s}$, and it is easy to verify that for any $s \geq 1$ and $t \geq 1$, $f$ is a decreasing function of $t$. Therefore, whenever $s \geq 1$ and $t \geq 1$, $$ f(t) = \left(1+\frac{1}{2st^2}\right)^t \leq 1+\frac{1}{2s}, $$ as required. 
\end{proof}

Next, we prove Theorem \ref{T3}. For natural numbers $s$ and $t$, define $$ F(s,t):=156s^2(t+1)^2(39s^2t^3)^t. $$

\begin{proof}[Proof of Theorem \ref{T3}]

For $N \geq (10s)^5 (t+1)^4 (39s^2t^3)^{5t}$, we construct an injective map $$ \eta: \{\pi \in D_{L,s,V} : |\pi| = N\} \rightarrow  \{\pi \in I_{L,s,V} : |\pi| = N \}. $$ Since $L \geq 2t+2$, the members of each pair in $H = \{(s+2u+1,s+2u+2): 0 \leq u \leq t\}$ lie in the set $\{s+1,\ldots, L+s\}$. Further, since $|H|=t+1 > t=|V|$, by the pigeonhole principle, at least one of the pairs $(s+2u+1,s+2u+2)$ has no member in $V$. Let $u_0$ be the least such number. To define the map $\eta$, we consider several cases. To prove that the map $\eta$ is injective, we ensure that for a partition $\pi \in D_{L,s,V}$, $\eta(\pi)$ has different frequencies of $s$ in different cases, and then prove that $\eta$ is injective within each case. Define the following subsets of $\mathbb{N}$.

\begin{align*}
V_1&:= \left\{n \in \mathbb{N}: \text{$n$ mod $(39s^2t^3)^t$} \in \left\{1,2,\ldots, (39s^2t^3)^t-(12st^3)^t\right\} \right\}, \\
V_2&:= \{(39s^2t^3)^t - \psi: 0 \leq \psi < (12st^3)^t\}, \\
V_3&:= \{2(39s^2t^3)^t - \psi: 0 \leq \psi < (12st^3)^t\}, \\
V_4&:= \{3(39s^2t^3)^t - \psi: 0 \leq \psi < (12st^3)^t\}, \\
V_5&:= \{4(39s^2t^3)^t - \psi: 0 \leq \psi < (12st^3)^t\}, \\
V_6&:= \{5(39s^2t^3)^t - \psi: 0 \leq \psi < (12st^3)^t\}, \\
V_7&:= \{6(39s^2t^3)^t - \psi: 0 \leq \psi < (12st^3)^t\}, \\
V_8&:= \{7h_0(39s^2t^3)^t - \psi: 0 \leq \psi < (12st^3)^t, h_0 \geq s+1\}.
\end{align*} 

It is easy to verify that all the above sets are disjoint. We ensure that in the various cases, the frequency of $s$ in $\eta(\pi)$ lies in one of these sets $V_i$, as described in Table \ref{tab2}.

\begin{table}[htpb]
    \centering
    \begin{tabular}{|c|c|}
 \hline
Case      &  The set $V_i$ containing possible frequencies of $s$ \\
\hline
1        &  $V_1$ \\
\hline
2(a)     &  $V_8$ \\
\hline
2(b)(i)       &   $V_2$  \\
\hline
2(b)(ii)     &  $V_3$  \\
\hline
2(b)(iii)(A)     &  $V_4$  \\
\hline
2(b)(iii)(B)     &  $V_5$  \\
\hline
2(b)(iii)(C)     &  $V_6$  \\
\hline
2(b)(iii)(D)     &  $V_7$  \\
\hline
    \end{tabular}
    \caption{The sets $V_i$ which contain the possible frequencies of $s$ in $\eta(\pi)$ in different cases}
        \label{tab2}
\end{table}

Case $1$: Suppose $\psi \geq (12st^3)^t$. We define the quantity $j(\psi)$ as follows.

\begin{equation}
\label{Crucial2}
j(\psi):= \left \lfloor \frac{\psi - (12st^3)^t}{(39s^2t^3)^t-(12st^3)^t} \right \rfloor. 
\end{equation}

That is, $j(\psi)$ is the unique integer satisfying the following equation. 

\begin{equation}
\label{Crucial2}
\left((39s^2t^3)^t-(12st^3)^t\right)j(\psi) + (12st^3)^t \leq \psi < \left((39s^2t^3)^t-(12st^3)^t\right)(j(\psi)+1) + (12st^3)^t.
\end{equation}

We claim that the following inequality holds.

\begin{equation}
\label{Crucial3}
k_1m_1^t + k_2m_2^t + \cdots + k_tm_t^t \geq s(\psi + (12st^3)^t (j(\psi)-1) + 1) + (s+2u_0)(s+2u_0+1). 
\end{equation}

To prove \eqref{Crucial3}, we begin by noting that 

\begin{align*}
(s+2u_0)(s+2u_0+1) &\leq (s+2t)(s+2t+1) \\
&= (s+2t)^2 + (s+2t) \\
&\leq (3st)^2 + (3st) \\
&\leq 12s^2t^2 \\
&\leq s(12st^3)^t. 
\end{align*}

Thus, to prove \eqref{Crucial3}, it suffices to prove the following inequality. 
\begin{equation}
\label{Crucial40}
k_1m_1^t + k_2m_2^t + \cdots + k_tm_t^t - s\psi \geq s(12st^3)^t j(\psi) + s. 
\end{equation}
To prove \eqref{Crucial40}, we begin by noting that
\begin{align}
\label{Crucial5}
 k_1m_1^t + k_2m_2^t + \cdots + k_tm_t^t & \geq (s+1) (m_1^t + m_2^t + \cdots m_t^t) \notag \\
 & \geq (s+1) (\max(m_1,m_2,\ldots,m_t))^t. 
 \end{align}
 From \eqref{Best2}, it follows that 
 \begin{equation} 
 \label{Crucial6}
 \max(m_1,m_2,\ldots,m_t) \geq \psi^{\frac{1}{t}} - 1.
 \end{equation}
 Using \eqref{Crucial5}, \eqref{Crucial6}, Lemma \ref{Crucial1} and the fact that $ \psi^{\frac{1}{t}} \geq 12st^3$, it follows that 
  \begin{align} 
 \label{Crucial7}
 k_1m_1^t + k_2m_2^t + \cdots + k_tm_t^t - s\psi & \geq (s+1) \left(\psi^{\frac{1}{t}} - 1 \right)^t - s\psi \notag \\
 &= s  \left(\psi^{\frac{1}{t}} - 1 \right)^t  \left( \left(1+\frac{1}{s}\right) - \left(1+\frac{1}{\psi^{\frac{1}{t}}-1}\right)^t \right) \notag \\
 &\geq \frac{(\psi^{\frac{1}{t}}-1)^t}{2}.
  \end{align}
 We first prove \eqref{Crucial40} in the case $j(\psi) = 0$. Substituting $j(\psi) = 0$ in \eqref{Crucial40} and using \eqref{Crucial7}, it suffices to prove that $$ \frac{(\psi^{\frac{1}{t}}-1)^t}{2} \geq s $$ which is true since $ \psi^{\frac{1}{t}} \geq 12st^3$. 
 Next, we prove  \eqref{Crucial40} in the case $j(\psi) \geq 1$. In this case, using \eqref{Crucial2}, it follows that 
 \begin{align}
 \label{Crucial337}
 \psi^{\frac{1}{t}} & \geq \left(\left((39s^2t^3)^t-(12st^3)^t\right)j(\psi)\right)^{\frac{1}{t}} \notag \\
   &\geq \left(\left((39s^2t^3)^t-(12s^2t^3)^t\right)j(\psi)\right)^{\frac{1}{t}} \notag \\
   &\geq \left((27s^2t^3)^t j(\psi)\right)^{\frac{1}{t}} \notag \\
   &\geq \left(27s^2t^3\right) (j(\psi))^{\frac{1}{t}}.
  \end{align}
 Then, using \eqref{Crucial337} and $j(\psi) \geq 1$, we have
 \begin{align}
 \label{Crucial8}
 \frac{(\psi^{\frac{1}{t}}-1)^t}{2} &\geq \frac{\left((27s^2t^3)(j(\psi))^{\frac{1}{t}}-1\right)^t}{2} \notag \\
 &\geq \frac{(26s^2t^3)^t j(\psi)}{2} \notag \\
 &\geq (13s^2t^3)^t j(\psi).
 \end{align}
 Using \eqref{Crucial7} and \eqref{Crucial8}, along with $j(\psi) \geq 1$, it follows that 
  \begin{align*} 
 \label{Crucial9}
k_1m_1^t + k_2m_2^t + \cdots + k_tm_t^t - s\psi &\geq (13s^2t^3)^t j(\psi) \notag \\
& \geq (12s^2t^3)^t j(\psi) + s \notag \\
& \geq s(12st^3)^t j(\psi) + s,
 \end{align*}
completing the proof of \eqref{Crucial40}, and thus also of \eqref{Crucial3}. Next, we use \eqref{Crucial3} to construct the injective map $\eta$. By \eqref{Crucial3} and Lemma \ref{Frobenius}, there exist nonnegative integers $x$ and $y$ such that the following equation holds.
\begin{equation}
\label{Crucial68}
k_1m_1^t + k_2m_2^t + \cdots + k_tm_t^t = s(\psi + (12st^3)^t (j(\psi)-1) + 1) + (s+2u_0+1)x + (s+2u_0+2)y. 
\end{equation}
For given $\psi$, the values of $m_1, m_2, \ldots, m_t$ are fixed. Then, fix a solution $(x_{\psi},y_{\psi})$ to \eqref{Crucial68}, and define the map $\eta$ as 
\begin{multline}
\label{Required1}
 \eta(\pi) = (s^{\psi + (12st^3)^t(j(\psi)-1) + 1}, \ldots, (s+2u_0+1)^{f_{s+2u_0+1}+x_{\psi}}, (s+2u_0+2)^{f_{s+2u_0+2}+y_{\psi}}, \\
  \ldots, k_t^0 , \ldots, k_1^0, \ldots),
  \end{multline}
where it is understood that the parts $s+2u_0+1$ and $s+2u_0+2$ are not precisely placed (it may, for example, be the case that $s+2u_0+1 > k_t$). To see the injectivity of $\eta$ in this case, note that the frequency of $s$ in $\eta(\pi)$ determines the value of $\psi$, which then fixes the values of $x_{\psi}$ and $y_{\psi}$. 

Case $2$: Suppose $\psi < (12st^3)^t$. We need to consider two subcases. 

Case $2$(a): Suppose there exists $h \in \mathbb{N}$ such that $s+1 \leq h \leq F(s,t)-1$ and $f_h \geq 8s(39s^2t^3)^t$. Let $h_0$ be the least such number. Then, by Lemma \ref{Frobenius}, there exist nonnegative integers $x$ and $y$ such that the following equation holds.
\begin{equation}
\label{Crucial101}
k_1m_1^t + k_2m_2^t + \cdots + k_tm_t^t + h_0(8s(39s^2t^3)^t) = s(7h_0(39s^2t^3)^t - \psi) + (s+2u_0+1)x + (s+2u_0+2)y. 
\end{equation}
For given $\psi$ and $h_0$, fix a solution $(x_{\psi,h_0},y_{\psi,h_0})$ to \eqref{Crucial101}. We define 
\begin{multline}
\label{Required2}
 \eta(\pi) = (s^{7h_0(39s^2t^3)^t-\psi}, \ldots, (s+2u_0+1)^{f_{s+2u_0+1}+x_{\psi,h_0}}, (s+2u_0+2)^{f_{s+2u_0+2}+y_{\psi,h_0}}, \\
  \ldots, k_t^0 , \ldots, k_1^0, \ldots, h_0^{f_{h_0}-8s(39s^2t^3)^t}, \ldots ). 
 \end{multline}
To see the injectivity of $\eta$ in this case, note that the frequency of $s$ in $\eta(\pi)$ determines the values of $\psi$ and $h_0$ (since $\psi < (12st^3)^t$), which then fix the values of $x_{\psi,h_0}$ and $y_{\psi,h_0}$.

Case $2$(b): Suppose for all $s+1 \leq h \leq F(s,t)-1$, $f_h < 8s(39s^2t^3)^t$. Since $N \geq (10s)^5 (t+1)^4 (39s^2t^3)^{5t}$ is large enough, there exists some $l \geq F(s,t)$ such that $f_l > 0$. Let $l_0$ be the least such number. For $1 \leq p \leq t+1$, define the following numbers.

\begin{align*}
\label{alpha}
&\alpha_p:= 5ps(39s^2t^3)^t+1 \notag \\
&\beta_p:=5ps(39s^2t^3)^t+2 \notag \\
&\gamma_p:=10ps(39s^2t^3)^t-1 \notag \\
&\delta_p:=15ps(39s^2t^3)^t-2.
\end{align*}

Since $|V|=t$, at least one of the tuples $(\alpha_p, \beta_p, \gamma_p, \delta_p)$ (as $p$ varies from $1$ to $t+1$) contains no members of $V$. Let $p_0$ be the least such value of $p$. Thus, $\alpha_{p_0}$, $\beta_{p_0}$, $\gamma_{p_0}$ and $\delta_{p_0}$ are not members of $V$.  

Case $2$(b)(i): Suppose $f_{\alpha_{p_0}} \geq 1$ and $f_{\gamma_{p_0}} \geq 1$. By Lemma \ref{Frobenius}, there exist nonnegative integers $x$ and $y$ such that the following equation holds.

\begin{equation}
\label{Crucial111}
15p_0s(39s^2t^3)^t + k_1m_1^t + k_2m_2^t + \cdots + k_tm_t^t = s((39s^2t^3)^t- \psi) + (s+2u_0+1)x + (s+2u_0+2)y. 
\end{equation}

For given $\psi$, fix a solution $(x_{\psi},y_{\psi})$ to \eqref{Crucial111}. We define 
\begin{multline}
 \eta(\pi) = (s^{(39s^2t^3)^t-\psi}, \ldots, (s+2u_0+1)^{f_{s+2u_0+1}+x_{\psi}}, (s+2u_0+2)^{f_{s+2u_0+2}+y_{\psi}}, \\
  \ldots, k_t^0 , \ldots, k_1^0, \ldots, \alpha^{f_{\alpha_{p_0}-1}}, \ldots, \gamma^{f_{\gamma_{p_0}-1}}, \ldots). 
 \end{multline}
 To see the injectivity of $\eta$ in this case, note that the frequency of $s$ in $\eta(\pi)$ determines the value of $\psi$, which then fixes the values of $x_{\psi}$ and $y_{\psi}$.
 
 Case $2$(b)(ii): Suppose $f_{\alpha_{p_0}} = 0$ or $f_{\gamma_{p_0}} = 0$. Further, suppose $f_{\beta_{p_0}} \geq 1$ and $f_{\delta_{p_0}} \geq 1$. By Lemma \ref{Frobenius}, there exist nonnegative integers $x$ and $y$ such that the following equation holds.

\begin{equation}
\label{Crucial11}
20p_0s(39s^2t^3)^t + k_1m_1^t + k_2m_2^t + \cdots + k_tm_t^t = s(2(39s^2t^3)^t- \psi) + (s+2u_0+1)x + (s+2u_0+2)y. 
\end{equation}

For given $\psi$, fix a solution $(x_{\psi},y_{\psi})$ to \eqref{Crucial11}. We define 
\begin{multline}
 \eta(\pi) = (s^{2(39s^2t^3)^t-\psi}, \ldots, (s+2u_0+1)^{f_{s+2u_0+1}+x_{\psi}}, (s+2u_0+2)^{f_{s+2u_0+2}+y_{\psi}}, \\
  \ldots, k_t^0 , \ldots, k_1^0, \ldots, \beta^{f_{\beta_{p_0}-1}}, \ldots, \delta^{f_{\delta_{p_0}-1}}, \ldots). 
 \end{multline}
 To see the injectivity of $\eta$ in this case, note that the frequency of $s$ in $\eta(\pi)$ determines the value of $\psi$, which then fixes the values of $x_{\psi}$ and $y_{\psi}$.

 Case $2$(b)(iii): Suppose $f_{\alpha_{p_0}} = 0$ or $f_{\gamma_{p_0}} = 0$, and $f_{\beta_{p_0}} = 0$ or $f_{\delta_{p_0}} = 0$. Then at least one of the following statements is true.
 
 \begin{itemize}
\item $T_1$: $f_{\alpha_{p_0}}=0$ and $f_{\beta_{p_0}}=0$;
\item $T_2$: $f_{\alpha_{p_0}}=0$ and $f_{\delta_{p_0}}=0$;
\item $T_3$: $f_{\gamma_{p_0}}=0$ and $f_{\beta_{p_0}}=0$;
\item $T_4$: $f_{\gamma_{p_0}}=0$ and $f_{\delta_{p_0}}=0$.
\end{itemize}
 
 Then, we have the following cases given below. Note that from the definition of the numbers $\alpha_p$, $\beta_p$, $\gamma_p$ and $\delta_p$, it is easy to verify that $\gcd(\alpha_p, \beta_p) = \gcd(\alpha_p, \delta_p) = \gcd(\beta_p, \gamma_p) = \gcd(\beta_p, \delta_p) = 1$. 
 
Case $2$(b)(iii)(A): Suppose $T_1$ is true. Since $l_0 \geq F(s,t)$ is large enough, by Lemma \ref{Frobenius}, there exist nonnegative integers $x$ and $y$ such that the following equation holds.

\begin{equation}
\label{Crucial120}
l_0+ k_1m_1^t + k_2m_2^t + \cdots + k_tm_t^t = s(3(39s^2t^3)^t- \psi) + \alpha_{p_0}x + \beta_{p_0}y. 
\end{equation}

For given $\psi$, fix a solution $(x_{\psi,l_0},y_{\psi,l_0})$ to \eqref{Crucial120}. We define 
\begin{equation}
 \eta(\pi) = (s^{3(39s^2t^3)^t-\psi}, \ldots, k_t^0 , \ldots, k_1^0, \ldots, \alpha_{p_0}^{x_{\psi,l_0}}, \beta_{p_0}^{y_{\psi,l_0}}, \ldots, l_0^{f_{l_0}-1}, \ldots). 
 \end{equation}

We describe the injectivity of the map $\eta$ in this subcase in detail. For the other subcases, the proof is similar, and will be skipped. The frequency of $s$ in $\eta(\pi)$ determines the value of $\psi$. Further, the frequencies of  $\alpha_{p_0}$ and  $\beta_{p_0}$ in $\eta(\pi)$ determine the values of $x_{\psi,l_0}$ and $y_{\psi,l_0}$. Then, one can find the value of $l_0$ using \eqref{Crucial120}, proving the required injectivity.

Case $2$(b)(iii)(B): Suppose $T_1$ is false and $T_2$ is true. Since $l_0 \geq F(s,t)$ is large enough, by Lemma \ref{Frobenius}, there exist nonnegative integers $x$ and $y$ such that the following equation holds.

\begin{equation}
\label{Crucial130}
l_0+ k_1m_1^t + k_2m_2^t + \cdots + k_tm_t^t = s(4(39s^2t^3)^t- \psi) + \alpha_{p_0}x + \delta_{p_0}y. 
\end{equation}

For given $\psi$ and $l_0$, fix a solution $(x_{\psi,l_0},y_{\psi,l_0})$ to \eqref{Crucial130}. We define
\begin{equation}
 \eta(\pi) = (s^{4(39s^2t^3)^t-\psi}, \ldots, k_t^0 , \ldots, k_1^0, \ldots, \alpha_{p_0}^{x_{\psi,l_0}}, \ldots, \delta_{p_0}^{y_{\psi,l_0}}, \ldots, l_0^{f_{l_0}-1}, \ldots). 
 \end{equation}
 
 Case $2$(b)(iii)(C): Suppose $T_1$ and $T_2$ are false, and $T_3$ is true. Since $l_0 \geq F(s,t)$ is large enough, by Lemma \ref{Frobenius}, there exist nonnegative integers $x$ and $y$ such that the following equation holds.

\begin{equation}
\label{Crucial140}
l_0+ k_1m_1^t + k_2m_2^t + \cdots + k_tm_t^t = s(5(39s^2t^3)^t- \psi) + \gamma_{p_0}x + \beta_{p_0}y. 
\end{equation}

For given $\psi$, fix a solution $(x_{\psi,l_0},y_{\psi,l_0})$ to \eqref{Crucial140}. We define
\begin{equation}
 \eta(\pi) = (s^{5(39s^2t^3)^t-\psi}, \ldots, k_t^0 , \ldots, k_1^0, \ldots, \beta_{p_0}^{y_{\psi,l_0}}, \ldots, \gamma_{p_0}^{x_{\psi,l_0}}, \ldots, l_0^{f_{l_0}-1}, \ldots). 
 \end{equation}

Case $2$(b)(iii)(D): Suppose $T_1$, $T_2$ and $T_3$ are false, and $T_4$ is true. Since $l_0 \geq F(s,t)$ is large enough, by Lemma \ref{Frobenius}, there exist nonnegative integers $x$ and $y$ such that the following equation holds.

\begin{equation}
\label{Crucial150}
l_0+ k_1m_1^t + k_2m_2^t + \cdots + k_tm_t^t = s(6(39s^2t^3)^t- \psi) + \gamma_{p_0}x + \delta_{p_0}y. 
\end{equation}

For given $\psi$, fix a solution $(x_{\psi,l_0},y_{\psi,l_0})$ to \eqref{Crucial150}. We define
\begin{equation}
 \eta(\pi) = (s^{6(39s^2t^3)^t-\psi}, \ldots, k_t^0 , \ldots, k_1^0, \ldots, \gamma_{p_0}^{x_{\psi,l_0}}, \ldots, \delta_{p_0}^{y_{\psi,l_0}}, \ldots, l_0^{f_{l_0}-1}, \ldots). 
 \end{equation}
\end{proof}

\section{An alternate proof of Theorem \ref{T3} for large $k_t$}
\label{S4}

Lemma \ref{Comb} and some of its proof ideas were suggested to the third author in a discussion in \cite{Stack1,Stack2}. However, we provide a detailed proof for the sake of completeness.

\begin{proof}[Proof of Lemma \ref{Comb}]
Consider the function $f:\mathbb{R^{+}} \to \mathbb{R^{+}}$ given by $f(x)=x^t$. Note that $f(x)$ is convex as $f''(x) \geq 0$ for all $x\in \mathbb{R^{+}}$. Then by the finite form of Jensen's inequality, we have
$$ f\left (\frac{m_1}{t}+\frac{m_2}{t}+\cdots + \frac{m_t}{t}\right ) \leq \frac{f(m_1)+f(m_2)+\cdots + f(m_t)}{t}. $$ Substituting the expression for $f$, we get 
\begin{equation}
\label{Jen1}
\frac{(m_1+m_2+\cdots + m_t)^t}{t^t} \leq \frac{m_{1}^t+m_{2}^t+ \cdots + m_{t}^t}{t}.
\end{equation}
On the other hand, it is clear that 
\begin{equation}
\label{Jen2}
 \binom{m_1 + m_2 + \cdots + m_t}{t} \leq \frac{(m_1+m_2+ \cdots + m_t)^t}{t!}.
\end{equation}
Further, using Pascal's formula, we have 
\begin{align}
\label{Jen3}
 \binom{m_1 + m_2 + \cdots + m_t}{t} &= \binom{m_1 + m_2 + \cdots + m_t-1}{t-1} + \binom{m_1 + m_2 + \cdots + m_t-1}{t} \notag \\
 &\geq \binom{m_1 + m_2 + \cdots + m_t-1}{t-1} \notag \\
 &\geq \binom{m_1 + m_2 + \cdots + m_{t-1}}{t-1}. 
\end{align}
Repeating the above procedure, it follows that 
\begin{equation}
\label{Jen4}
\binom{m_1 + m_2 + \cdots + m_t}{t} \geq \binom{m_1 + m_2 + \cdots + m_{t-1}}{t-1} \geq \cdots \geq \binom{m_1}{1}. 
\end{equation}
From \eqref{Jen2} and \eqref{Jen4}, we have
\begin{equation}
\label{Jen5}
{m_1 + m_2 + \cdots + m_t \choose t} + \cdots + {m_1 + m_2 \choose 2} + {m_1 \choose 1} \leq \frac{t(m_1+m_2+ \cdots + m_t)^t}{t!}.
\end{equation}
The required inequality now follows immediately from \eqref{Jen1} and \eqref{Jen5}.
\end{proof}

Suppose $V = \{k_1, \ldots k_t\}$ with $k_1 > k_2 > \cdots > k_t$. Further suppose $k_t \geq \left(2^{t+4}s+s^2\right) \frac{t^t}{t!}$. We use Lemma \ref{Comb} to prove Theorem \ref{T3} in this case. We generalize the ideas used in the proof of \cite[Theorem 9]{BR20}, and recall some notation.
For any $s \geq 1$, define the quantities:

\begin{itemize}
\item $F(s) = (10s-2)(15s-3)+8s$; 
\item $\kappa(s) = (12s-1)((s+1)+(s+2)+\cdots (F(s)-1))+1$.
\end{itemize}

We note in passing that $\kappa(s)<(15s)^5$. To prove Theorem \ref{T3}, for all $N \geq \kappa(s)$, we construct an injective map $$ \eta: \{\pi \in D_{L,s,V} : |\pi| = N\} \rightarrow  \{\pi \in I_{L,s,V} : |\pi| = N \}. $$  Let $\pi = ((s+1)^{f_{s+1}}, \ldots, k_t^{f_t}, \ldots, k_1^{f_1}, \ldots, (L+s)^{f_{L+s}}) \in D_{L,s,V}$. From the definition of $D_{L,s,V}$, we have $f_i = m_i^t$ for some $m_i \geq 0$. To construct the map $\eta$, we need to consider two cases.

Case $1$: Suppose $m_i = 0$ for all $i$. In this case, we do not have any parts of $k_i$ to remove. However we need to add parts of $s$, and compensate by removing parts of some other elements of $\pi$. This case is then essentially same as the case $f_k = 0$ of the proof of Theorem \ref{MostGen}. Therefore, in this case, the proof described in \cite[Case $1$ of Theorem 9]{BR20} can be applied verbatim. Thus, the possible frequencies of $s$ in the image set could be $2,4,6,8,15,20$ or multiples of $12$, as described in \cite[Table 1]{BR20}.

Case $2$: Suppose some of the $m_i$'s are non-zero. Let $m_{i_1}, m_{i_2}, \ldots, m_{i_p}$ be the non-zero ones with $i_1 < i_2 < \cdots < i_p$, while the others are zero. We associate the tuple $(m_1,m_2, \ldots, m_t)$ to the number $\gamma(m_1,m_2,\ldots,m_t)$, whose binary expansion is obtained by writing $0$ at the $i^{th}$ place if $m_i=0$, and $1$ otherwise. For example, if all the $m_i$'s are non-zero, then $\gamma(m_1,m_2,\ldots,m_t)$ has the binary number representation $1,1,1,\ldots,1$. That is, $\gamma(m_1,m_2,\ldots,m_t) = 2^t-1$. For brevity of notation, we denote $\gamma(m_1,m_2,\ldots,m_t)$ by $\gamma$. Thus, in this notation, the value of $\gamma$ uniquely determines the set of all $i$ such that $m_i \neq 0$. By Lemma \ref{Comb}, we have

\begin{align}
\label{Cruciall}
k_1m_1^t + k_2m_2^t + \cdots + k_tm_t^t & \geq  k_t \left(m_1^t + m_2^t + \cdots + m_t^t\right) \notag \\
& \geq  k_t \left(m_{i_1}^t + m_{i_2}^t + \cdots + m_{i_p}^t\right) \notag \\
& \geq \left(2^{t+4}s+s^2\right) \left(m_{i_1}^t + m_{i_2}^t + \cdots + m_{i_p}^t\right)  \frac{t^t}{t!} \notag \\
 & \geq \left(2^{t+4}s+s^2\right) \left({m_{i_1} + m_{i_2} + \cdots + m_{i_p} \choose p} + \cdots +  {m_{i_1} \choose 1}\right) \notag \\
 & \geq 2^{t+4}s  \left({m_{i_1} + m_{i_2} + \cdots + m_{i_p} \choose p} + \cdots +  {m_{i_1} \choose 1}\right)  + s^2 \notag \\
 & \geq 2^{t+4}s \left({m_{i_1} + m_{i_2} + \cdots + m_{i_p} \choose p} + \cdots +  {m_{i_1} \choose 1}-1\right) + s(s+1).
\end{align}

By \eqref{Cruciall} and Lemma \ref{Frobenius}, there exist nonnegative integers $x$ and $y$ such that the following equation holds. 

\begin{multline}
\label{Cruciall2}
k_1m_1^t + k_2m_2^t + \cdots + k_tm_t^t = s\left(2^{t+4}\left({m_{i_1} + m_{i_2} + \cdots + m_{i_p} \choose p} + \cdots +  {m_{i_1} \choose 1}\right) - (2\gamma - 1)\right) \\
+ (s+1)x + (s+2)y. 
\end{multline}

For given $m_1$, $m_2$, \ldots, $m_t$, fix some nonnegative integer solution $(x,y)$ of \eqref{Cruciall2}. Then, we define the map $\eta$ as follows.

\begin{multline*}
\eta(\pi) = \Bigg(s^{\left(2^{t+4}\left({m_{i_1} + m_{i_2} + \cdots + m_{i_p} \choose p} + \cdots +  {m_{i_1} \choose 1}\right) - (2\gamma - 1)\right)}, (s+1)^{f_{s+1}+x}, (s+2)^{f_{s+2}+y}, \\
\ldots, k_t^{0}, \ldots, k_1^{0}, \ldots, (L+s)^{f_{L+s}}) \Bigg).
\end{multline*}

The frequency of $s$ modulo $2^{t+4}$ in $\eta(\pi)$ determines the value of $\gamma$, which in turn determines the set of values of $i$ for which $m_i \neq 0$. That is, we know the tuple $(i_1,i_2,\ldots,i_p)$. Then, the frequency of $s$ gives the value of $$ {m_{i_1} + m_{i_2} + \cdots + m_{i_p} \choose p} + \cdots +  {m_{i_1} + m_{i_2} \choose 2} + {m_{i_1} \choose 1}, $$ which determines the values of $m_{i_1}, m_{i_2}, \ldots, m_{i_p}$ because of the injectivity of the map in \eqref{Oneone}. The values of other $m_i$ are $0$. Thus, we know the values of all the $m_i$'s. Then, the values of $x$ and $y$ are also known, as decided by \eqref{Cruciall2}. Therefore, the frequencies of $s+1$ and $s+2$ in $\eta(\pi)$ determine the values of $f_{s+1}$ and $f_{s+2}$ as well. Thus, the map $\eta$ is injective in Case $2$. 

For overall injectivity, note that the frequency of $s$ in $\eta(\pi)$ in Case $2$ is an odd number greater than $15$. To verify this statement, note that $\gamma < 2^t$, and thus
\begin{align*}
2^{t+4}\left({m_{i_1} + m_{i_2} + \cdots + m_{i_p} \choose p} + \cdots +  {m_{i_1} \choose 1}\right) - (2\gamma - 1) &> 2^{t+4} - 2^{t+1} \\
&= 7(2^{t+1}) \\
&\geq 28.
\end{align*}

\section{q-series analogues of Theorem \ref{T1} and Theorem \ref{T3}}
\label{S5}
Berkovich and Uncu \cite[Section 7]{BU19} also considered the $q$-series analogues of Theorem \ref{MostGen}. We recall their notation. Define

\begin{itemize}
    \item the \emph{$q$-Pochhammer symbol} by $$(a;q)_n := (1-a)(1-aq)(1-aq^2)
        \cdots (1-aq^{n-1}),$$
		for an integer $n \geq 1$, with $(a;q)_0 := 1$;
\item the series $H_{L,s,k}(q)$ by
    \begin{equation*}
        H_{L,s,k}(q) := \frac{q^s(1-q^k)}{(q^s;q)_{L+1}} -
        \left(\frac{1}{(q^{s+1};q)_L}-1\right),
    \end{equation*}
    for positive integers $L,s$ and $k$. 
\end{itemize}

Then, Theorem \ref{MostGen} asserts that for $s+1 \leq k \leq L+s$, the $q$-series $H_{L,s,k}$ is \emph{eventually positive}, meaning that all its coefficients are positive after a certain bound. Presently for an impermissible set $V = \{k_1,k_2,\ldots,k_t\} \subset \{s+1,\ldots,L+s\}$, we define the series $H_{L,s,V}$ as follows. Define

\begin{itemize}
\item the series $H_{L,s,V}(q)$ by
    \begin{equation*}
        H_{L,s,V}(q) := \frac{q^s(1-q^{k_1})(1-q^{k_2})\cdots(1-q^{k_t})}{(q^s;q)_{L+1}} -
        \left(\frac{1}{(q^{s+1};q)_L}-1\right),
    \end{equation*}
\end{itemize}

Then, Theorem \ref{T1} implies that for $t \geq 2$, $H_{L,s,V}$ is \emph{eventually negative}, meaning that all its coefficients are negative after a certain bound. The bound is a polynomial in $L$ and $s$. To see a $q$-series analogue of Theorem \ref{T3}, we define the $q$-series $H'_{L,s,V}$ as follows. 

\begin{itemize}
\item Define the series $H'_{L,s,V}(q)$ by
    \begin{multline*}
        H'_{L,s,V}(q) := \frac{q^s(1-q^{k_1})(1-q^{k_2})\cdots(1-q^{k_t})}{(q^s;q)_{L+1}}  \\
        - \frac{(1-q^{k_1})(1-q^{k_2})\cdots(1-q^{k_t}) \sum_{i \geq 0} q^{k_1i^t} \sum_{i \geq 0} q^{k_2i^t} \cdots  \sum_{i \geq 0} q^{k_ti^t}}{(q^{s+1};q)_L}.
    \end{multline*}
    
    Then Theorem \ref{T3} implies that for $t \leq \frac{L-2}{2}$, the $q$-series $H'_{L,s,V}$ is eventually positive, with the bound after which the coefficients are positive being a polynomial in $s$ (for a given $t$). This leads to another interesting consequence. It is easy to verify that the product of two eventually positive $q$-series is eventually positive. Thus, the $q$-series $$\frac{H'_{L,s,V}}{(1-q^{k_1})(1-q^{k_2})\cdots(1-q^{k_t})}$$ is also eventually positive. That is, the series $H''_{L,s,V}$ defined as 
      \begin{equation*}
        H''_{L,s,V}(q) := \frac{q^s}{(q^s;q)_{L+1}}  \\
        - \frac{\sum_{i \geq 0} q^{k_1i^t} \sum_{i \geq 0} q^{k_2i^t} \cdots  \sum_{i \geq 0} q^{k_ti^t}}{(q^{s+1};q)_L}.
    \end{equation*}
      is eventually positive. To understand the combinatorial interpretation of the eventual positivity of $H''_{L,s,V}$, we define the set $P_{L,s,V}$ as follows.
\end{itemize}

\begin{definition}
For any $V \subset \{s+1, \ldots L+s\}$, $P_{L,s,V}$ consists of partitions with parts in the set $\{s+1, \ldots L+s\}$, such that members of $V$ can appear in two colours-green and red, where the red parts appear with a perfect $|V|^{th}$ power frequency. \end{definition}

Then, the eventual positivity of $H''_{L,s,V}$ for $|V| \leq \frac{L-2}{2}$ means that there exists a bound $M$ such that whenever $N \geq M$, the number of partitions that have smallest part $s$ and largest part $\leq L+s$ is greater than the number of partitions in $P_{L,s,V}$. This result is closely linked to \cite[Theorem 4.1]{Thesis} and can be viewed as an improvement over it as $L \rightarrow \infty$. 

\section{Concluding remarks}
\label{S6}

In this article, we have shown that for a given set $V \subset \{s+1, \ldots L+s\}$, such that $2 \leq |V| \leq \frac{L-2}{2}$, there exists a bound $M$ such that whenever $N \geq M$, the number of partitions of $N$ in $I_{L,s,V}$ is bounded above by those in $D_{L,s}$, and bounded below by those in $D_{L,s,V} \subset D_{L,s}$. One possible future direction is to study what happens if for $|V| \geq 2$, instead of comparing $I_{L,s,V}$ with $D_{L,s,V}$, we compare $I_{L,s,V}$ with $E_{L,s,V} \subset D_{L,s}$, where $E_{L,s,V}$ is defined as follows. 

\begin{itemize}
\item $E_{L,s,V}$  denotes the set of nonempty partitions
		with parts in the set $\{s+1,  \ldots,
		L+s\}$; such that the frequencies of the members of $V$ are perfect $(t-1)^{th}$ powers $(0,1, 2^{t-1}, \ldots)$, where $t = |V|$.
\end{itemize}

 \begin{question}
    \label{Q5}
    Suppose $V \subset \{s+1, \ldots L+s\}$ is not a singleton set. Does there exist a bound $M$, such that for $N \geq M$, 
	 \begin{equation*}
        |\{\pi \in I_{L,s,V} : |\pi| = N \}| \leq |\{\pi \in E_{L,s,V} :
        |\pi| = N\}|.
    \end{equation*}

    \end{question}
    
    If the assertion in Question \ref{Q5} is true, it would imply that the number of partitions in $I_{L,s,V}$ lies between the number of partitions in $D_{L,s,V}$ and $E_{L,s,V}$. That is, the number of partitions in $I_{L,s,V}$ is bounded above by partitions in $D_{L,s}$ with frequencies of members of $V$ perfect $t^{th}$ powers, and bounded below by partitions in $D_{L,s}$ with frequencies of members of $V$ perfect $(t-1)^{th}$ powers.
    
    For $t=2$, $E_{L,s,V} = D_{L,s}$, and thus Question \ref{Q5} has been proved in the affirmative by Theorem \ref{T1}. However, for $t \geq 3$, this seems a hard problem. We focus our attention on the case $t=3$, and rephrase Question \ref{Q5} as follows.

    A natural way to approach Question \ref{Q5} is to generalize the methods in the proof of Theorem \ref{T1}. Basically one would need to find analogues of \eqref{Later}, \eqref{Later'}, \eqref{Later''} and \eqref{Later'''}. For $t=3$, the primary question that one needs to answer is whether every sufficiently large number $n$ can be expressed in the form $k_3x^2 + k_2y^2 + k_1z^2$ for some integers $x,y,z$. In other words, we need to know whether the ternary diagonal quadratic form $k_3x^2 + k_2y^2 + k_1z^2$ is almost universal. This is also in general a hard question and is an active area of research (see \cite{Bochnak, PW, WS, Sun, HK}). The current knowledge \cite{HK} suggests that most of these diagonal quadratic forms are not almost universal, but there exist infinitely many which are almost universal. It will be interesting to see if one could make some progress on Question \ref{Q5} using universality or almost $(k,l)$ universality of these quadratic forms.
    
 Another possible direction is to compare the set $I_{L,s,V}$ with the set $D_{L,s,V'}$ where $V'$ is obtained by removing an element of $V$. We make this question precise. 
    
\begin{question}
    \label{Q6}
    Suppose $V = \{k_1, k_2, \ldots, k_t\} \subset \{s+1, \ldots L+s\}$ is such that $V$ contains an odd prime p which does not divide any other element of $V$. Does there exist a bound $M$, such that for $N \geq M$, 
	 \begin{equation*}
        |\{\pi \in I_{L,s,V} : |\pi| = N \}| \leq |\{\pi \in D_{L,s,V \setminus \{p\}} :
        |\pi| = N\}|.
    \end{equation*}

    \end{question}

Though Question \ref{Q6} in general remains open, we briefly describe how one could approach Question \ref{Q6} in the special case $t=2$. Suppose $V = \{k_1,k_2,k_3\}$ be such that $k_1$ is an odd prime and it does not divide $k_2$ and $k_3$. We need to appropriately modify the proof of Theorem \ref{T1} by finding analogues of \eqref{Later}, \eqref{Later'}, \eqref{Later''} and \eqref{Later'''}. To find the analogues to these equations, we basically need to show that every sufficiently large number can be expressed in the form $k_3x^2 + k_2y^2 + k_1z$. 

From the theory of quadratic residues, we know that there are $\frac{k_1+1}{2}$ quadratic residues mod $k_1$ (including $0$). Therefore, mod $k_1$, as $x$ and $y$ vary, the expressions $n-k_3x^2$ and $k_2y^2$ can take $\frac{k_1+1}{2}$ values each. Since there are only $k_1$ numbers modulo $k_1$, by the pigeonhole principle, the expressions $n-k_3x^2$ and $k_2y^2$ must match for some $x$ and $y$. That is, there exist integers $x$ and $y$ such that $$ n \equiv k_3x^2 + k_2y^2 \bmod k_1. $$ Clearly we can choose $x<k_1$ and $y< k_1$. Thus, for $n \geq k_3(k_1^2+k_2^2)$, there exists a nonnegative integer $z$ such that $$ n = k_3x^2 + k_2y^2 + k_1z $$ has a solution $(x,y,z)$. 

It will be interesting to see if one can use theory of higher power residues, instead of quadratic residues, to make progress on Question \ref{Q6} in general. 

\section{Acknowledgements}

The first author (now affiliated to SLIET Longowal) wishes to thank IISER Mohali, while the second and the third authors thank Shiv Nadar University for providing funding and research facilities. The authors also thank Kavita Reckwar for some initial discussions. 

\section{Conflict of Interest}

On behalf of all authors, the corresponding author states that there is no conflict of interest. 

\section{Data availability}

All data generated or analysed during this study are included in this article.

\end{document}